\documentclass[11pt]{article}

\usepackage{amsmath}
\usepackage{amssymb}
\usepackage{amscd}
\usepackage{amsthm}
\usepackage{indentfirst}

\usepackage{tikz}
\usepackage{wrapfig}

\newcommand{\Z}{\ensuremath{\mathbb{Z}}}
\newcommand{\Q}{\ensuremath{\mathbb{Q}}}
\newcommand{\C}{\ensuremath{\mathbb{C}}}
\newcommand{\F}{\ensuremath{\mathbb{F}}}

\usepackage[margin=2.6cm]{geometry}

\theoremstyle{plain}
\newtheorem{thm}{Theorem}[section]
\newtheorem{lem}[thm]{Lemma}

\theoremstyle{definition}

\newtheorem{rmq}[thm]{Remark}
\newtheorem{exm}[thm]{Example}

\DeclareMathOperator{\Spec}{Spec}
\DeclareMathOperator{\Pic}{Pic}

\DeclareMathOperator{\Gal}{Gal}
\DeclareMathOperator{\Res}{Res}
\DeclareMathOperator{\NS}{NS}

\DeclareMathOperator{\rk}{rk}
\DeclareMathOperator{\car}{char}
\DeclareMathOperator{\GL}{GL}
\DeclareMathOperator{\SL}{SL}

\DeclareMathOperator{\Cl}{Cl}
\DeclareMathOperator{\Disc}{Disc}

\newcommand{\Gm}{\mathbf{G}_{{\rm m}}}
\newcommand{\E}{\mathcal{E}}
\newcommand{\Qbar}{\overline{\Q}}

\begin{document}

\title{Descent on elliptic surfaces and arithmetic bounds for the Mordell-Weil rank}

\author{Jean Gillibert \and Aaron Levin}

\date{June 2021}

\maketitle

\begin{abstract}
We introduce the use of $p$-descent techniques for elliptic surfaces over a perfect field of characteristic not $2$ or $3$. Under mild hypotheses, we obtain an upper bound for the rank of a non-constant elliptic surface. When $p=2$, this bound is an arithmetic refinement of a well-known geometric bound for the rank deduced from Igusa's inequality. This answers a question raised by Ulmer. We give some applications to rank bounds for elliptic surfaces over the rational numbers.
\end{abstract}


\section{Introduction}

The aim of this paper is to introduce the use of classical $p$-descent techniques for elliptic curves over function fields of one variable, and to use these techniques to derive general bounds for the Mordell-Weil rank which are sensitive to the field of constants of the function field.  Let $k$ be a field of characteristic not $2$ or $3$ and let $S$ be a smooth projective geometrically integral curve over $k$.  Let $E$ be a (nonconstant) elliptic curve over the function field $k(S)$.  Using cohomological arguments (see \S{}\ref{intro:Igusa}), one obtains the following well-known upper bound for the rank:
\begin{equation}
\label{eq:Igusa}
\rk_{\Z} E(k(S))\leq 4g(S)-4+\deg(\mathfrak{f}_E),
\end{equation}
where $g(S)$ is the genus of $S$ and $\deg(\mathfrak{f}_E)$ is the degree of the conductor of $E$, viewed as a divisor on $S$ (in fact, \eqref{eq:Igusa} also holds in characteristics $2$ and $3$).  Following \cite{silverman04}, we call \eqref{eq:Igusa} the \emph{geometric rank bound}.  This bound is geometric in the sense that the right-hand side does not depend on the field $k$, and in particular the inequality \eqref{eq:Igusa} holds with $k$ replaced by its algebraic closure.  In positive characteristic, there are examples with conductor of arbitrarily large degree for which the inequality \eqref{eq:Igusa} is sharp (see \cite{shioda86}, \cite{ulmer02}).

When $k$ is a number field, Ulmer \cite[Section 9]{ulmer04} has raised the question of the existence of \emph{arithmetic bounds} improving \eqref{eq:Igusa}, that is, refinements of the inequality \eqref{eq:Igusa} which depend on the arithmetic of the field $k$.  Ulmer noted that work of Silverman \cite{silverman00, silverman04} on ranks of elliptic curves over abelian towers provided some evidence for the existence of such a bound.  Under mild hypotheses, we give a positive answer to Ulmer's question, providing a bound for the rank which is equivalent to \eqref{eq:Igusa} when $k$ is algebraically closed (see \S{}\ref{intro:Igusa}), and which in general yields an improvement depending on the arithmetic of the field $k$.  The precise statement is given in Theorem \ref{thm_kummer}: for each suitable prime $p$, we obtain a bound which depends crucially on the number of $k$-rational $p$-torsion points in the Jacobian of a certain curve associated to the $p$-torsion subgroup of $E$.
Two examples (\ref{exm:tomasz}, \ref{exm2}), and one application (Theorem~\ref{thm:0bound3}), are given when $k=\Q$.

When $k=\F_q$ is a finite field, Brumer \cite{brumer92} used Weil's ``explicit formula" to prove an arithmetic bound for the rank of $E$ depending on $q$:
\begin{equation}
\label{eq:Brumer}
\rk_{\Z} E(k(S))\leq \frac{4g(S)-4+\deg(\mathfrak{f}_E)}{2\log_q \deg(\mathfrak{f}_E)}+ c \frac{\deg(\mathfrak{f}_E)}{(\log_q \deg(\mathfrak{f}_E))^2},
\end{equation}
where $c$ is an explicit constant depending only on $g(S)$ and $q$.  Although there are examples with conductor of arbitrarily large degree for which the main term in the inequality \eqref{eq:Brumer} is sharp (see \cite{ulmer02}), our bounds will frequently also provide an improvement to \eqref{eq:Brumer}.  Silverman \cite[Conjecture 4]{silverman04} has stated a conjectural analogue of Brumer's inequality  over number fields, and  Theorem~\ref{thm_kummer} provides a possible approach towards Silverman's conjecture (Remark \ref{rm:SilConjecture}).

It is quite surprising that, although $p$-descent techniques have been extensively used in order to bound the rank of elliptic curves over number fields, we have not been able to track a similar use of these techniques in the function field setting, apart from a few exceptions (see Remark~\ref{rmq:Selmer}).  

Instead, a more common approach takes advantage of the rich theory that can be developed in this setting by relating the geometry of surfaces and the Mordell-Weil group of an elliptic curve over a function field, enriched with the lattice structure on the torsion-free part induced by the N{\'e}ron-Tate height pairing.  The detailed analysis of these Mordell-Weil lattices was developed simultaneously by Elkies and Shioda in the late 80's, and quickly became a  powerful and central tool in the study of elliptic curves over function fields.
 
Despite the success of this approach, we believe that $p$-descent techniques can shed new light on classical results in the function field setting, and provide new ideas to tackle open questions. Additionally, these techniques may be advantageous for studying ``arithmetic'' problems, where one desires finer information over a non-algebraically closed field.  For instance, in recent work we give applications of our general $p$-descent result to the construction of number fields with ``large'' ideal class groups \cite{gl19}, and in a forthcoming paper, we give applications to the study of integral points on elliptic curves over function fields.


\subsection{$p$-descent: the generic case}
\label{intro:descent}

Let $k$ be a perfect field of characteristic not $2$ or $3$.  In the applications we have in mind, $k$ may be a number field, or a finite field, or the algebraic closure of such fields.

Let $S$ be a smooth projective geometrically integral curve over $k$, and let $k(S)$ be the function field of $S$. By abuse of notation, we identify closed points of the scheme $S$ and discrete valuations of $k(S)/k$. If $v$ is such a valuation, we denote by $k_v$ the residue field of $v$, which is a finite extension of $k$.
The scheme $S$ being a Dedekind scheme, if we start with an elliptic curve over $k(S)$ we may consider its N{\'e}ron model over $S$. In the present paper, the word \emph{N{\'e}ron model} refers to N{\'e}ron's group scheme model, which is the smooth locus of N\'eron's minimal regular model (see \cite[\S{}1.5]{NeronModels} or \cite[\S{}10.2]{liu}).

Throughout this paper, we consider the following setting:

\begin{enumerate}
\item $E$ is an elliptic curve over $k(S)$.
\item $\mathcal{E}\to S$ is the N\'eron model of $E$ over $S$.
\item $\Sigma\subset S$ is the set of places of bad reduction of $\E$.
\item $p\neq \car(k)$ is a prime number.
\item $\mathcal{E}[p]\to S$ is the group scheme of $p$-torsion points of $\E$. The map $\mathcal{E}[p]\to S$ is {\'e}tale, and its restriction to $S\setminus \Sigma$ is finite of degree $p^2$.
\end{enumerate}

The \'etaleness of $\mathcal{E}[p]\to S$ stated above follows from the fact that the prime $p$ is invertible on $S$ (assumption 4); see \cite[\S{}7.3, Lemma~2, (b)]{NeronModels} for a proof. Since $S$ is smooth and the composition of smooth morphisms is smooth, one deduces that $\E[p]\setminus \{0\}$  is smooth over $k$.

If $v\in \Sigma$ is a place of bad reduction of $\E$, we denote by $\E_v:=\E\times_S k_v$ the special fiber of $\E$ at $v$, by $\E_v^0$ the connected component of the identity in $\E_v$, and by $\Phi_v$ the component group of $\E_v$. By definition, we have an exact sequence for the {\'e}tale topology on $k_v$
\begin{equation}
\label{eq:compgroup}
\begin{CD}
0 @>>> \E_v^0 @>>> \E_v @>>> \Phi_v @>>> 0, \\
\end{CD}
\end{equation}
and $\Phi_v$ is a finite {\'e}tale group scheme over $k_v$. The \emph{Tamagawa number} of $\E$ at $v$ is by definition the order of $\Phi_v(k_v)$, and we denote it by $c_v$. Finally, we denote by $\Phi$ the skyscraper sheaf over $S$ whose fiber at $v$ is $\Phi_v$, and by $\E^{p\Phi}$ the inverse image of $p\Phi$ by the natural map $\E\to\Phi$.

The scheme $\E[p]\setminus \{0\}$ (where $\{0\}$ denotes the image of the unit section of $\mathcal{E}\to S$) is a smooth scheme of dimension one over $k$, with a degree $p^2-1$ quasi-finite map to $S$. The deficiency of finiteness of this map can be understood as follows: above a place $v$ of multiplicative reduction at which $\Phi_v$ has no $p$-torsion, the fiber of $\E[p]\to S$ has $p$ points instead of $p^2$. Above a place $v$ of additive reduction at which $\Phi_v$ has no $p$-torsion, the fiber of $\E[p]\to S$ has only one point. As a result, the scheme $\E[p]$ is not projective over $k$.

We shall denote by $C$ the smooth compactification of $\E[p]\setminus \{0\}$, endowed with its canonical finite map $C\to S$ of degree $p^2-1$. By construction, $C$ is a smooth projective curve over $k$, and the inclusion $\E[p]\setminus \{0\} \hookrightarrow C$ is an isomorphism above the open subset of good reduction of $E$.

By elementary Galois theory, the following conditions are equivalent:
\begin{enumerate}
\item[(i)] $C$ is a geometrically integral $k$-curve;
\item[(ii)] $(E[p]\otimes_k \overline{k})\setminus \{0\}$ is the spectrum of a field;
\item[(iii)] the action of $\Gal(\overline{k(S)}/\overline{k}(S))$ on $E[p]\setminus \{0\}$ is transitive.
\end{enumerate}

In the sequel, we shall assume that these conditions hold. This implies that $E$ is non-constant, but does not prevent it from being isotrivial. On the other hand, if $E$ is not isotrivial, then it was proved by Igusa \cite{Igusa59} that, for all but finitely many $p$, the image of the Galois representation $\Gal(\overline{k(S)}/\overline{k}(S))\to \GL_2(\F_p)$ attached to $E[p]$ is $\SL_2(\F_p)$, which implies condition (iii).

We are ready to state the main result of this paper, which provides an arithmetic upper bound on the rank of elliptic curves over function fields.
The proof, which is given in \S{}\ref{sub:descent}, relies on $p$-descent techniques, analogous to the number field case.

\begin{thm}
\label{thm_kummer}
Assume that the action of $\Gal(\overline{k(S)}/\overline{k}(S))$ on $E[p]\setminus \{0\}$ is transitive, and let $C$ be the smooth compactification of $\E[p]\setminus \{0\}$. Then:
\begin{enumerate}
\item[1)] If $p\geq 3$, there is an injective morphism
\begin{equation}
\label{eq:pdescent_map}
\E^{p\Phi}(S)/p\E(S) \longrightarrow  \ker\left(N_{C/C^+}:\Pic(C)[p]\to\Pic(C^+)[p]\right),
\end{equation}
where $C^+$ is the quotient of $C$ by the involution $P\mapsto -P$, and  $N_{C/C^+}$ denotes the norm map. It follows that
\begin{equation}
\label{inequality1}
\rk_{\Z} E(k(S)) \leq \dim_{\F_p} \Pic(C)[p] - \dim_{\F_p} \Pic(C^+)[p] + \#\{v\in \Sigma, p\mid c_v\}.
\end{equation}
where $c_v$ denotes the Tamagawa number of $\E$ at $v$.
\item[2)] If $p=2$, there is an injective morphism
\begin{equation}
\label{eq:2descent_map}
\E^{2\Phi}(S)/2\E(S) \longrightarrow  \ker\left(N_{C/S}:\Pic(C)[2]\to\Pic(S)[2]\right).
\end{equation}
It follows that
\begin{equation}
\label{inequality2}
\begin{split}
\rk_{\Z} E(k(S)) \leq &\dim_{\mathbb{F}_2} \Pic(C)[2] - \dim_{\mathbb{F}_2} \Pic(S)[2] + \#\{v\in \Sigma, 2\mid c_v\} \\
& +\#\{v\in \Sigma, \text{the red. type of $\E$ at $v$ is $\mathrm{I}_{2n}^*$ for some $n\geq 0$} \}.
\end{split}
\end{equation}
\end{enumerate}
\end{thm}

A few comments on the statement of Theorem~\ref{thm_kummer}:
\begin{enumerate}
\item The group $\E^{p\Phi}(S)$ is the group of sections of $\E$ which, in each fiber, reduce to a component which is a multiple of $p$ in the group of components. The quotient $\E^{p\Phi}(S)/p\E(S)$ measures the difference between sections which are locally multiples of $p$ and those which are globally multiples of $p$ (where locally should be understood in the sense of {\'e}tale topology). The injective morphism \eqref{eq:pdescent_map} is, roughly speaking, a coboundary map for the Kummer exact sequence induced by multiplication by $p$ on $\E$.
\item Let $J_C$ be the Jacobian of $C$. Then $\Pic(C)[p]$ is a subgroup of $J_C(k)[p]$, the group of $k$-rational $p$-torsion points on $J_C$ (equality holds, for example, if $C(k)\neq\emptyset$). Thus, $\Pic(C)[p]$ is a quantity which depends on the arithmetic of $C$ over $k$.
\item Observe that the norm map $N_{C/C^+}:\Pic(C)[p]\to\Pic(C^+)[p]$ is surjective when $p\geq 3$, because the degree of $C\to C^+$ is $2$, which is coprime to $p$. It follows that the quantity $\dim_{\mathbb{F}_p} \Pic(C)[p] - \dim_{\mathbb{F}_p} \Pic(C^+)[p]$ cannot decrease when enlarging the field $k$. The same remark applies when $p=2$, in which case $C\to S$ has degree $3$, and hence the norm map $N_{C/S}:\Pic(C)[2]\to\Pic(S)[2]$ is surjective.
\item In the bound \eqref{inequality1}, the terms $\dim_{\mathbb{F}_p} \Pic(C)[p] - \dim_{\mathbb{F}_p} \Pic(C^+)[p]$ and $\#\{v\in \Sigma, p\mid c_v\}$ are both of arithmetic nature. As we have seen, the first one depends on the size of the $k$-rational $p$-torsion subgroups of the Jacobians of the curves $C$ and $C^+$.
The second one is the number of closed points of $S$ at which $\E$ has bad reduction and Tamagawa number divisible by $p$. In fact, it follows from Lemma~\ref{cm} that
$$
\#\{v\in \Sigma, p\mid c_v\}=\dim_{\F_p} H^0(S,\Phi[p]).
$$
When enlarging $k$, this number can increase for two reasons: firstly, a closed point which is not $k$-rational can split as the sum of several points over a larger field, which has the effect of enlarging the set $\Sigma$; secondly, a component of order $p$ in the group $\Phi_v$ which is not $k_v$-rational can become rational over a larger field. The same remark applies, \emph{mutatis mutandis}, for $p=2$.

\item In the case when $p=2$, the bound \eqref{inequality2} is a geometric analogue of the classical bound of Brumer and Kramer \cite[Proposition~7.1]{BK77} for the rank of an elliptic curve over $\Q$ (which does not have a rational point of order $2$) in terms of the $2$-torsion of the class group of a cubic field and bad reduction data.
\item A comment on the terminology: when we say that $E$ has a fiber of type $\mathrm{I}_{2n}^*$ at $v$, we mean it over $k_v$, and not just over $\overline{k}$. More precisely, this means that the Kodaira type of $\E_v$ over $\overline{k}$ is $\mathrm{I}_{2n}^*$, and that the four components of $\E_v$ are rational over $k_v$, in other terms $\Phi_v(k_v)\simeq (\Z/2\Z)^2$. In general, the reduction type at $v$ can be described by the data of the reduction type over $\overline{k}$ together with the action of the absolute Galois group of $k_v$ on $\Phi_v$. See \cite[\S{}10.2]{liu}.
\item In practice, given a Weierstrass equation $y^2=x^3+Ax+B$ for $E$, with $A,B\in k(S)$, one can obtain explicit defining equations for $C$ and $C^+$ over $S$. More precisely, if $p=2$ then $C\to S$ is the degree $3$ cover defined by the equation $x^3+Ax+B=0$. If $p\geq 3$, then $C^+\to S$ is the degree $(p^2-1)/2$ cover defined by the vanishing of the $p$-division polynomial $\psi_p$ of $E$, and the curve $C$ is the double cover of $C^+$ defined by $y^2=x^3+Ax+B$, with $\psi_p(x)=0$.
\item In specific examples, computing the bound \eqref{inequality1} requires an efficient algorithm for computing the $p$-torsion in the Jacobian of a curve. Over finite fields, polynomial-time algorithms exist (see \cite{couveignes09}). Over number fields the situation is more complicated, but by reduction modulo a good prime $\neq p$ one obtains an upper bound on the $p$-torsion.
\end{enumerate}

\begin{rmq}
\label{rmq:pdescentimprovement}
When $p\geq 5$, one can improve \eqref{eq:pdescent_map} as follows: we have an injective morphism
$$
\E^{p\Phi}(S)/p\E(S) \longrightarrow \bigcap_{C\to C'} \ker\left(N_{C/C'}:\Pic(C)[p]\to\Pic(C')[p]\right)
$$
where $C\to C'$ runs through all proper subcovers of $C\to S$. See Remark~\ref{rmq:Dokchitser} for the details.
\end{rmq}

\begin{rmq}
It is possible to generalize the statement above by replacing $E[p]$ by a group $G\subset E$ of order $p>2$ satisfying conditions similar to (i)--(iii). By considering the isogeny $\lambda:E\to F$ with kernel $G$, one can prove an analogous result in which $\E^{p\Phi}(S)/p\E(S)$ is replaced by $\mathcal{F}^{\lambda\Phi}(S)/\lambda\E(S)$, with obvious notation. Note that $p=2$ is excluded here, because we need the sum of all elements of $G(\overline{k(S)})$ to be zero in order to make Lemma~\ref{w_injective} work.
\end{rmq}

\begin{rmq}
\label{rmq:Selmer}
The proof of Theorem~\ref{thm_kummer} relies on the computation of a geometric analogue of the $p$-Selmer group, under the assumption that $\Gal(\overline{k(S)}/\overline{k}(S))$ acts transitively on $E[p]\setminus \{0\}$.

There are a few occurrences in the literature of geometric analogues of Selmer groups. Let us cite in particular \cite{CSSD98}, whose Section~4 contains definitions and various properties of these groups; as we note in Remark~\ref{rmq:geometricSelmer}, the {\'e}tale cohomology groups that we compute are closely related to them. We also refer the reader to \cite{ellenberg06}, which is more focused on $p^\infty$-Selmer groups and fundamental groups.
\end{rmq}

\begin{rmq}
\label{rmq:nongeneric}
When the full $p$-torsion of $E$ is defined over $k(S)$, one can prove with the same techniques that there exists an injective morphism
$$
\E^{p\Phi}(S)/p\E(S) \rightarrow H^1(S,\mu_p)^2,
$$
where
$$
H^1(S,\mu_p) \simeq (k^{\times}/(k^{\times})^p) \oplus\Pic(S)[p].
$$

In this case, if $k$ is a number field, then $k^{\times}/(k^{\times})^p$ is infinite, and so the natural generalization of Theorem~\ref{thm_kummer} is not relevant (bounding a finite number by $+\infty$). This provides examples in which the geometric analogue of the Selmer group is infinite.

On the other hand, if $k$ is algebraically closed, then $k^{\times}/(k^{\times})^p=\{1\}$, and we obtain the upper bound
$$
\dim_{\F_p} \E^{p\Phi}(S)/p\E(S) \leq 4g(S),
$$
where $g(S)$ denotes the genus of $S$. The rationality of the full $p$-torsion puts strong constraints on the reduction type of $\E$, and in particular, $p\mid c_v$ for each place $v$ of bad reduction. It follows that, in this case, the natural analogues of the inequalities \eqref{inequality1} and \eqref{inequality2} are both equivalent, whatever the value of $p$ is, to 
$$
\rk_{\Z} E(k(S))\leq 4g(S)-2+\deg(\mathfrak{f}_E),
$$
a bound which is weaker than the geometric rank bound \eqref{eq:Igusa}.
\end{rmq}


\subsection{Refinements of the geometric rank bound}
\label{intro:Igusa}

Let us assume that $k$ is algebraically closed. In this case, there are well-known bounds for the rank of elliptic curves over $k(S)$ in terms of the conductor of the curve and the genus of $S$. For a nice overview of this topic, we refer the reader to \cite{shioda92}.

Let $\overline{\E}$ denote the minimal regular model of $E$ over $S$, which is also a smooth compactification of the $k$-surface $\E$, and let $\NS(\overline{\E}):=\Pic(\overline{\E})/\Pic^0(\overline{\E})$ be the N{\'e}ron-Severi group of $\overline{\E}$, which is a free $\Z$-module of finite rank. We let
$$
\rho :=\rk_\Z \NS(\overline{\E}),
$$
the so-called Picard number of $\overline{\E}$. Igusa's inequality asserts that
$$
\rho \leq b_2,
$$
where $b_2$ is the second Betti number of the surface $\overline{\E}$. If the characteristic of $k$ is not zero, $b_2$ is computed using $\ell$-adic cohomology, with $\ell\neq \car(k)$.

Let $\mathfrak{f}_E:=\sum_{v\in\Sigma} f_v\cdot v$ be the conductor of $\E$, which is a divisor on $S$. It is well-known \cite[Chap.~IV, Theorem~10.2]{silvermanII} that $f_v=1$ if the reduction type of $\E$ at $v$ is multiplicative and $f_v=2$ if the reduction is additive; the wild part of the conductor is zero since $\car(k)\neq 2, 3$.

The Grothendieck-Ogg-Shafarevich formula \cite{raynaud95} states that
\begin{equation}
\label{eq:GOS}
b_2-\rho=4g(S)-4-\rk_\Z E(k(S)) + \deg(\mathfrak{f}_E),
\end{equation}
where $\deg(\mathfrak{f}_E)=\sum_{v\in \Sigma}f_v$ is the degree of the divisor $\mathfrak{f}_E$. Therefore, Igusa's inequality is equivalent to the geometric bound:
\begin{equation*}
\rk_{\Z} E(k(S))\leq 4g(S)-4+\deg(\mathfrak{f}_E).
\end{equation*}

In \S{}\ref{sub:Igusa}, we prove the following.

\begin{thm}
\label{thm:arithmeticIgusa}
Assume that $E(\overline{k}(S))[2]=0$, and let $C$ be the smooth compactification of $\E[2]\setminus\{0\}$. Then
\begin{equation}
\label{eq:2descentkbar}
\begin{split}
4g(S)-4+\deg(\mathfrak{f}_E) = 2g(C) &- 2g(S)  + \#\{v\in \overline{\Sigma}, 2\mid c_v\} \\
& +\#\{v\in \overline{\Sigma}, \text{the red. type of $\E$ at $v$ is $\mathrm{I}_{2n}^*$ for some $n\geq 0$} \}.
\end{split}
\end{equation}
It follows that the bound \eqref{inequality2} is a refinement of the geometric rank bound \eqref{eq:Igusa}.
\end{thm}

Thus, as mentioned in the introduction, we obtain a positive answer to Ulmer's question from the survey \cite[\S{}9]{ulmer04}. This refinement allows one to improve on known bounds even in characteristic zero (see Example~\ref{exm:tomasz} and Theorem \ref{thm:0bound3}).

As a matter of fact, another refinement occurs via $3$-descent, provided the field $k$ does not contain cube roots of unity.

\begin{thm}
\label{thm:3arithmeticIgusa}
Assume that the prime $3$ satisfies (i)--(iii). Let $C$ be the smooth compactification of $\E[3]\setminus\{0\}$, and let $C^+$ be the quotient of $C$ by the involution $P\mapsto -P$. Then
\begin{equation}
\label{eq:3descentkbar}
4g(S)-4+\deg(\mathfrak{f}_E) = g(C) - g(C^+)  + \#\{v\in \overline{\Sigma}, 3\mid c_v\}.
\end{equation}
It follows that, for $p=3$, the bound \eqref{inequality1} is a refinement of the geometric rank bound \eqref{eq:Igusa} provided the field $k$ does not contain cube roots of unity.
\end{thm}

\begin{rmq}
In fact, Theorem~\ref{thm_kummer} provides a family of upper bounds on the rank of $E$ over $k(S)$, one for each prime $p\neq \car(k)$ satisfying assumptions (i)--(iii). While over $\overline{k}$, the bound for $p=2$ is equivalent to the geometric rank bound \eqref{eq:Igusa},  the same phenomenon does not hold for $p=3$, as Theorem~\ref{thm:3arithmeticIgusa} shows:  the difference between the rank bound coming from $3$-descent and the geometric rank bound is equal to $g(C) - g(C^+)$ over $\overline{k}$. In fact, the genus of $C$ increases with $p$, and so for large $p$ the bound \eqref{inequality1} will be weaker than the bound \eqref{eq:Igusa} over $\overline{k}$ (also note that when $p\geq 5$ our $p$-descent bound may be improved as in Remark~\ref{rmq:Dokchitser}). However, over nonalgebraically closed fields $k$, it may still be the case that \eqref{inequality1} provides an improved bound, depending on the size of the rational $p$-torsion subgroup of the Jacobian of $C$ (see Remark \ref{rm:SilConjecture}).
\end{rmq}

\begin{rmq}
When $k$ is an algebraically closed field of positive characteristic, the bound \eqref{eq:Igusa} can be improved under the assumption that the formal Brauer group of $\E$ has finite height $h$, in which case it was proved by Artin and Mazur \cite{AM77} that $b_2-\rho\geq 2h$. This result, which relies on computations in crystalline cohomology, is called the Igusa-Artin-Mazur inequality. As with Igusa's inequality, this is a result of geometric nature.
\end{rmq}

We end this section by discussing the case when $k$ has characteristic zero, and in particular, the case when $k$ is a number field.  When $k$ is an algebraically closed field of characteristic zero, the bound \eqref{eq:Igusa} can be improved as follows:  Lefschetz's inequality states that $\rho\leq h^{1,1}$, and hence it follows from Hodge theory that $b_2-\rho\geq 2p_g$, where $p_g$ is the geometric genus of the surface $\E$. By Grothendieck-Ogg-Shafarevich \eqref{eq:GOS}, we conclude that
\begin{equation}
\label{eq:0bound}
\rk_{\Z} E(k(S))\leq 4g(S)-4 +\deg(\mathfrak{f}_E) -2p_g.
\end{equation}

In contrast to the bound \eqref{eq:Igusa} in positive characteristic, Lefschetz's bound \eqref{eq:0bound} is not known to be sharp when the right-hand side is large, even over the field of complex numbers.  In fact,  we do not even know whether or not there exist non-constant elliptic curves over $\C(t)$ with arbitrarily large rank (the current record, due to Shioda \cite{shioda92}, is a curve of rank $68$ over $\C(t)$).

As the next example shows, when working over $\Q(t)$, the inequality \eqref{inequality2} allows one in practice to improve on the bound \eqref{eq:0bound}, provided one is able to compute the rational $2$-torsion in the Jacobian of a certain trigonal curve.

\begin{exm}
\label{exm:tomasz}
Let $q\geq 5$ be a prime number, and let $a\in\Z$ be an odd integer. Let $E$ be the elliptic curve over $\Q(t)$ defined by the equation
$$
y^2 = x^3 + t^q + a.
$$
Since the polynomial $t^q+a$ is square-free, it follows from Tate's algorithm \cite{tate75} that $E$ has additive reduction of type $\mathrm{II}$ at roots of $t^q+a$, and additive reduction of type $\mathrm{II}$ or $\mathrm{II}^*$ at infinity (because $q\equiv 1 ~\text{or}~ 5\pmod{6}$). The degree of the conductor of $E$ is $2(q+1)$.
One also computes that the geometric genus is $p_g=\left\lfloor \frac{q-1}{6} \right\rfloor$. Hence, according to \eqref{eq:0bound}, we have the bound
$$
\rk_{\Z} E(\Qbar(t)) \leq 2q-2-2\left\lfloor \frac{q-1}{6} \right\rfloor.
$$

Let us now compute the bound \eqref{inequality2} from Theorem~\ref{thm_kummer}. The Tamagawa numbers of $E$ are all equal to $1$, hence
$$
\rk_{\Z} E(\Q(t)) \leq \dim_{\F_2} \Pic(C)[2],
$$
where $C$ is the curve defined by the equation $x^3=t^q+a$. According to \cite[Theorem~2.6]{tomasz16}, the Jacobian of $C$ has no rational $2$-torsion, the integer $a$ being odd. Therefore, \eqref{inequality2} yields
$$
\rk_{\Z} E(\Q(t)) =0.
$$
On the other hand, the bound \eqref{inequality2} over $\Qbar$ is equivalent to the geometric rank bound:
$$
\rk_{\Z} E(\Qbar(t)) \leq 2q-2.
$$
\end{exm}

\begin{exm}
\label{exm2}
Let $\beta\in\Q^\times$ and let $E$ be the elliptic curve over $\Q(t)$ defined by
$$
y^2=x^3+tx^2+t^2(t^2+\beta).
$$
The singular fibers of this curve are listed below (we assume that $\beta\neq 2^2\cdot 3^{-6}$ so that $27(t^2+\beta)+4t$ is not a square over $\Qbar$):
\begin{center}
\begin{tabular}{r|cccc}
& $t^2+\beta=0$ & $27(t^2+\beta)+4t=0$ & $t=0$ & $t=\infty$ \\
\hline
fiber type & $\mathrm{I}_1$ & $\mathrm{I}_1$ & $\mathrm{IV}$ & $\mathrm{IV}$ \\
order of $\Phi_v$ & $1$ & $1$ & $3$ & $3$ \\
\end{tabular}
\end{center}

Its conductor is
$$
\mathfrak{f}_E = \{t^2+\beta=0\} + \{27(t^2+\beta)+4t=0\} + 2\cdot (\{0\}+ \{\infty\})
$$
which has degree $8$.
Moreover, $\E$ is a rational elliptic surface, hence has geometric genus $p_g=0$, and so the bounds \eqref{eq:Igusa} and \eqref{eq:0bound} agree and yield:
$$
\rk_{\Z} E(\Qbar(t)) \leq \deg(\mathfrak{f}_E) - 4 = 4.
$$

In fact, this bound is sharp and the exact value of the rank over $\Qbar(t)$ is $4$. More generally, the structure of the Mordell-Weil lattice and related information for $\E$ can be found in entry No.~11 in the table of Oguiso and Shioda \cite{OS}, who classified Mordell-Weil lattices of rational elliptic surfaces over algebraically closed fields.

On the other hand, let $C$ be the curve defined by the equation $x^3+tx^2+t^2(t^2+\beta)=0$. The Tamagawa numbers being odd, none of the bad fibers contributes to the arithmetic bound \eqref{inequality2}, and hence we have
$$
\rk_{\Z} E(\Q(t)) \leq \dim_{\mathbb{F}_2} \Pic(C)[2].
$$

By substituting $X=x/t$ and $Y=2t+X^3+X^2$, we find that $C$ is a hyperelliptic curve of genus $2$, with equation
$$
Y^2=X^4(X+1)^2-4\beta.
$$

Let $s$ be the number of irreducible factors of $X^4(X+1)^2-4\beta$ over $\Q$; then
$$
\dim_{\mathbb{F}_2} \Pic(C)[2] \leq \dim_{\mathbb{F}_2} J_C(\Q)[2] =
\begin{cases} s-1 & \text{if all factors of $X^4(X+1)^2-4\beta$ have even degree} \\ s-2 &\text{otherwise} \end{cases}
$$

We compute some examples in the following table:
\begin{center}
\begin{tabular}{r|cccccc}
$\beta$ & $2^25^43^{-12}$ & $2^23^47^{-6}$ & $1$ & $9$ & ${-3^4}5^42^{-8}$ & $2$  \\
\hline
factorization type of $X^4(X+1)^2-4\beta$ & $[1,1,2,2]$ & $[1,1,1,3]$ & $[1,2,3]$ & $[3,3]$ & [2,4] & $[6]$  \\
rank bound over $\Q(t)$ & $2$ & $2$ & $1$ & $0$ & $1$ & $0$ \\
\end{tabular}
\end{center}
In fact, it can be shown that these examples exhaust all of the possible factorization types for $\beta\in\Q^\times\setminus \{2^2\cdot 3^{-6}\}$, and it's not hard to parametrize each possibility. For instance, the first four factorization types can occur only when $\beta$ is a perfect square, and the factorization type $[1,1,2,2]$ occurs precisely when $\beta=\frac{(\alpha^3+\alpha)^4(\alpha-1)^2}{4(\alpha^3+1)^6}$ for some $\alpha\in \mathbb{Q}$.
\end{exm}

In a recent preprint \cite[Theorem~7.1]{bhargava17}, Bhargava \emph{et al.} have given the following bounds for $2$-torsion over finite fields:
\begin{thm}
\label{thm:Bhargava}
Let $C$ be a smooth projective curve of genus $g$ over $\mathbb{F}_q$.  Then
\begin{align*}
\dim_{\F_2} \Pic(C)[2] \leq \log_2\left(\frac{q^{g+1}-1}{q-1}\right).
\end{align*}
If $C$ admits a degree $n$ map to $\mathbb{P}^1$ (over $\mathbb{F}_q$) then
\begin{align*}
\dim_{\F_2} \Pic(C)[2] \leq \left(1-\frac{1}{n}\right)g\log_2 q+O_n(1).
\end{align*}
\end{thm}
We note that these results only improve on the trivial bound $\dim_{\F_2} \Pic(C)[2]\leq 2g(C)$ when $q$ is very small. For example, if $C$ is trigonal, then we should take $q\leq 7$. If $C$ is a curve over a number field $k$, then we consequently obtain bounds for $\dim_{\F_2} \Pic(C)[2]$ whenever $C$ has good reduction at a prime of small (odd) norm.  Thus, when $k$ is a number field, Theorem \ref{thm:Bhargava} leads to improvements of the inequalities \eqref{eq:Igusa} and \eqref{eq:0bound} under suitable hypotheses.  Let $\chi=1-g(S)+p_g$, where $p_g$ is the geometric genus of the surface $\E$.  Then (still assuming $\car(k)=0$) we have the well-known inequalities \cite{shioda92}
\begin{align*}
2\chi+2-2g(S)\leq \deg(\mathfrak{f}_E)\leq 12\chi.
\end{align*}

In particular, in terms of the invariant $\chi$, the bound \eqref{eq:0bound} yields another well-known inequality for the rank when $\car(k)=0$:
\begin{equation}
\label{eq:0bound2}
\rk_{\Z} E(k(S))\leq 10\chi +2g(S)-2.
\end{equation}
As a sample application, under suitable good reduction hypotheses, in Section \ref{sub:0bound} we prove the following asymptotic improvement to \eqref{eq:0bound2} over the rational numbers.
\begin{thm}
\label{thm:0bound3}
Let $E$ be an elliptic curve over $\Q(S)$ and suppose that $E(\Qbar(S))[2]=0$. Let $C$ be the smooth compactification of $\E[2]\setminus\{0\}$. 
\begin{enumerate}
\item[1)] If $C$ has good reduction at $3$, then
\begin{align}
\label{imp1}
\rk_{\Z} E(\Q(S))\leq 6 (\log_23)\chi+3(\log_23)g(S)-(\log_23-1)\sum_{v\in\overline{\Sigma}}\varepsilon_v-\log_23-1,
\end{align}
where the $\varepsilon_v$ are defined in Lemma~\ref{lem:C2ramif}.

\item[2)] If $\Q(S)=\mathbb{Q}(t)$ and $C$ has good reduction at $p\in \{3,5\}$, then
\begin{align}
\label{imp2}
\rk_{\Z} E(\Q(t))\leq 4(\log_2p)\chi +O(1).
\end{align}
\end{enumerate}
\end{thm}

\begin{rmq}
If $E(\Qbar(S))[2]\neq 0$, then Cox \cite{cox82} proved the better (geometric) bound:
\begin{align*}
\rk_{\Z} E(\Qbar(S))\leq 6\chi+2g(S)-2.
\end{align*}
\end{rmq}

\begin{rmq}
The numerical expansions of the coefficients of $\chi$ in Theorem \ref{thm:0bound3} are:
\begin{align*}
6 (\log_23)&=9.509775\cdots\\
4 (\log_23)&=6.339850\cdots\\
4 (\log_25)&=9.287712\cdots.
\end{align*}
\end{rmq}

If $E$ is an elliptic curve defined over $\Q(t)$, there is a strategy which allows one, in specific cases, to improve on Lefschetz's bound \eqref{eq:0bound} over $\Qbar(t)$. One picks a suitable prime $p$ of good reduction for the surface $\E$, and one computes the characteristic polynomial of the Frobenius acting on the second $\ell$-adic cohomology group of the reduced surface over $\overline{\F}_p$. The details of this approach are explained in \cite[\S{}6]{vanluijk07} and \cite[\S{}4]{kloosterman07}. Obviously, one may replace $\Q$ by a number field $k$. Nevertheless, this technique is not ``arithmetic'', in the sense that the bound it provides is valid over $\Qbar$.

The question of refining Lefschetz's bound over $\Q$, and more generally over number fields, has been addressed previously by various authors, see the discussion in \cite[\S{}13.12]{SchSch10}. In this respect, to our knowledge, only specific elliptic surfaces (e.g., K3 surfaces)  have been studied so far. So, it seems to us that Theorem~\ref{thm_kummer} is the first general result which provides an upper bound of an ``arithmetic nature'' on the rank of elliptic curves over function fields.

We end with a speculative remark on using our results to prove statements towards Silverman's conjectural analogue of \eqref{eq:Brumer} when the base field $k$ is a number field instead of a finite field.

\begin{rmq}
\label{rm:SilConjecture}
Silverman \cite[Conjecture 4]{silverman04} has conjectured that if $k$ is a number field and $E$ is a non-isotrivial elliptic curve over $k(S)$, then
\begin{align}
\label{eq:Sil}
\rk_{\Z} E(k(S))\ll \frac{\deg(\mathfrak{f}_E)}{\log \deg(\mathfrak{f}_E)},
\end{align}
where the implied constant depends only on $k$ and $S$ (in fact, Silverman states a more precise conjecture).

To approach \eqref{eq:Sil} using Theorem~\ref{thm_kummer}, one needs strong bounds on the rank of $p$-torsion in $\Pic(C)$.  A possible argument for the existence of such bounds is as follows.  Let $p$ be a prime and $d$ a positive integer. Brumer and Silverman \cite{BS} have raised the question of the existence of a constant $c_{p,d}$ such that for every number field $k/\Q$ of degree $d$,
\begin{align}
\label{eq:class}
\dim_{\F_p} \Cl(k)[p]\leq c_{p,d}\frac{\log |\Disc k|}{\log\log |\Disc k|},
\end{align}
where $\Cl(k)$ is the ideal class group of $k$, and $\Disc k$ is the discriminant of $k$.

Let us fix a finite field $\mathbb{F}_q$.  Curves of genus $g$ and gonality $d$ over $\mathbb{F}_q$ are analogous to number fields of degree $d$ over $\mathbb{Q}$ with discriminant (roughly) $q^{2g}$.  Thus, a function field analogue of \eqref{eq:class} might assert the existence of a constant $c_{p,d,q}$ such that
\begin{align}
\label{eq:Fp}
\dim_{\F_p} \Pic(C)[p]\leq c_{p,d,q}\frac{g}{\log 2g},
\end{align}
where $C$ is any curve over $\mathbb{F}_q$ of positive genus $g$ admitting a morphism of degree $d$ to $\mathbb{P}^1$.  When $p\mid d$, the example of hyperelliptic\footnote{More precisely, by considering an equation $y^2=f(x)$ where $f\in\Q[x]$ splits as a product of $2g+1$ distinct linear factors, one obtains a hyperelliptic curve $C$ over $\Q$ for which $\dim_{\F_2}\Pic(C)[2]=2g$.} (or more generally superelliptic) curves shows that \eqref{eq:Fp} cannot be extended from finite fields to number fields $k$.  However, if $p\nmid d$, it seems possible that \eqref{eq:Fp} may continue to hold when $k$ is a number field.

Now let $k$ be a number field and let $E$ and $C$ be as in Theorem~\ref{thm_kummer} (for the prime $p$).  We have a natural morphism $C\to S$ of degree $p^2-1$, which is unramified outside of the set of places of bad reduction $\Sigma$.  We can bound the degree of the ramification divisor by $(p^2-2)\deg(\mathfrak{f}_E)$.  Assume additionally that not more than $\frac{\deg(\mathfrak{f}_E)}{\log \deg(\mathfrak{f}_E)}$ Tamagawa numbers are divisible by $p$.  Then assuming that \eqref{eq:Fp} holds for number fields when $p\nmid d$ (note that $p$ and $p^2-1$ are coprime), Theorem~\ref{thm_kummer} immediately implies that
\begin{align*}
\rk_{\Z} E(k(S))\ll \frac{\deg(\mathfrak{f}_E)}{\log \deg(\mathfrak{f}_E)},
\end{align*}
where the implied constant depends on $k$, $S$, and also $p$.  Thus, we would obtain Silverman's conjecture for various large families of elliptic curves (one family for each fixed prime $p$).
\end{rmq}


\section{Proofs}


\subsection{Proof of Theorem~\ref{thm_kummer}}
\label{sub:descent}


We denote by $\E^0$ the open subgroup of $\E$ whose fiber at each $v$ is $\E_v^0$, and we call it (by abuse of notation) the connected component of $\E$.

By globalizing \eqref{eq:compgroup}, we obtain an exact sequence for the {\'e}tale topology on $S$
$$
\begin{CD}
0 @>>> \E^0 @>>> \E @>>> \Phi:=\bigoplus_{v\in \Sigma} (i_v)_*\Phi_v @>>> 0, \\
\end{CD}
$$
where $i_v:\Spec(k_v)\to S$ denotes the canonical inclusion.

We denote by $\E^{p\Phi}$ the open subgroup scheme of $\E$ which is the inverse image of $p\Phi$ by the quotient map above, so that we have an exact sequence for the {\'e}tale topology on $S$
\begin{equation*}
\begin{CD}
0 @>>> \E^{p\Phi} @>>> \E @>>> \Phi/p\Phi @>>> 0. \\
\end{CD}
\end{equation*}
Applying global sections to this sequence, we obtain an exact sequence
$$
\begin{CD}
0 @>>> \E^{p\Phi}(S) @>>> \E(S) @>>> H^0(S,\Phi/p\Phi). \\
\end{CD}
$$

Since the group $p\E(S)$ is a subgroup of $\E^{p\Phi}(S)$, if we mod out the first two groups by $p\E(S)$ we obtain another exact sequence
\begin{equation}
\label{Ep_phi2}
\begin{CD}
0 @>>> \E^{p\Phi}(S)/p\E(S) @>>> \E(S)/p\E(S) @>>> H^0(S,\Phi/p\Phi). \\
\end{CD}
\end{equation}

Before we proceed with the proof of Theorem~\ref{thm_kummer}, we shall prove a few lemmas.

\begin{lem}
\label{cm}
Let $v\in \Sigma$ be a place of bad reduction for $\E$.
\begin{enumerate}
\item[1)] If $p\geq 3$, then
$$
\dim_{\F_p} H^0(k_v,\Phi_v[p]) = \dim_{\F_p}  H^0(k_v,\Phi_v/p\Phi_v)=\left\{
\begin{array}{ll}
1 & \text{if $p\mid c_v$}\\
0 & \text{otherwise}.
\end{array}\right.
$$
\item[2)] If $p=2$, then
$$
\dim_{\F_2} H^0(k_v,\Phi_v[2]) = \dim_{\F_2} H^0(k_v,\Phi_v/2\Phi_v)=\left\{
\begin{array}{ll}
2 & \text{if reduction type $\mathrm{I}_{2n}^*$ for some $n\geq 0$}\\
1 & \text{if $2\mid c_v$, other reduction type}\\
0 & \text{otherwise}.
\end{array}\right.
$$
\end{enumerate}
\end{lem}

\begin{proof}
It follows from the explicit description of the N{\'e}ron model of an elliptic curve (see \cite[\S{}10.2, Remark~2.24]{liu} or \cite[Chap.~IV, Table~4.1]{silvermanII}) that, if $p\geq 3$, or if $p=2$ and the reduction type of $\E$ at $v$ is not $\mathrm{I}_{2n}^*$ for some $n\geq 0$, then the $p$-primary component of $\Phi_v(\overline{k_v})$ is cyclic. In this situation, it follows from basic cohomology that the following conditions are equivalent:
\begin{enumerate}
\item the Galois module $\Phi_v/p\Phi_v$ is either trivial, or else a non-constant Galois module over $k_v$;
\item $H^0(k_v,\Phi_v/p\Phi_v)=0$;
\item $p\nmid c_v=\# \Phi_v(k_v)$.
\end{enumerate}
This proves the lemma in all cases except when $p=2$ and the reduction type is $\mathrm{I}_{2n}^*$ for some $n\geq 0$. In this case, $\Phi_v$ is isomorphic to the constant $k_v$-group scheme $(\Z/2\Z)^2$, hence the result.
\end{proof}

\begin{lem}
\label{rank_size}
\begin{enumerate}
\item[1)] If $p\geq 3$, we have
$$
\dim_{\F_p} \E^{p\Phi}(S)/p\E(S) \geq \rk_{\Z} E(k(S)) - \#\{v\in \Sigma, p\mid c_v\}.
$$
\item[2)] If $p=2$, we have
\begin{equation*}
\begin{split}
\dim_{\F_2} \E^{2\Phi}(S)/2\E(S) \geq & \rk_{\Z} E(k(S)) - \#\{v\in \Sigma, 2\mid c_v\} \\
& -\#\{v\in \Sigma, \text{the red. type of $\E$ at $v$ is $\mathrm{I}_{2n}^*$ for some $n\geq 0$} \}.
\end{split}
\end{equation*}
\end{enumerate}
\end{lem}

\begin{proof}
It follows from the exact sequence \eqref{Ep_phi2} that
$$
\dim_{\F_p} \E^{p\Phi}(S)/p\E(S) \geq \dim_{\F_p} \E(S)/p\E(S) - \dim_{\F_p} H^0(S,\Phi/p\Phi).
$$
By the universal property of the N{\'e}ron model, $\E(S)/p\E(S)=E(k(S))/pE(k(S))$. Because $E[p]$ satisfies (i)--(iii), $E(k(S))$ has no nontrivial $p$-torsion, hence
$$
\dim_{\F_p} \E(S)/p\E(S) = \rk_{\Z} E(k(S)).
$$
On the other hand,
$$
H^0(S,\Phi/p\Phi) = \bigoplus_{v\in\Sigma} H^0(k_v,\Phi_v/p\Phi_v),
$$
and Lemma~\ref{cm} allows one to compute the $\F_p$-dimension of each of these groups.
\end{proof}

Let $j:\eta=\Spec(k(S))\to S$ be the inclusion of the generic point of $S$. Let us recall that, if $H$ is a $k(S)$-group scheme, then by definition the N{\'e}ron model of $H$ is a smooth separated $S$-group scheme of finite type which represents the sheaf $j_*H$ on the smooth site of $S$. Therefore, we have an exact sequence
$$
0 \longrightarrow H^1(S,j_*H) \longrightarrow H^1(k(S),H) \longrightarrow H^0(S,R^1j_*H) \longrightarrow \cdots.
$$

In particular, if $G$ is a N{\'e}ron model of its generic fiber, then the restriction to the generic fiber induces an injection
\begin{equation}
\label{injNeron}
H^1(S,G) \hookrightarrow H^1(k(S),G)
\end{equation}

\begin{rmq}
The image of \eqref{injNeron} can be described as the set of Galois cohomology classes which are unramified everywhere. We shall not make use of this description, but it may be helpful to keep in mind.
\end{rmq}

\begin{lem}
\label{unramified}
\begin{enumerate}
\item[1)] The group scheme $\E[p]$ is the N{\'e}ron model of $E[p]$.
\item[2)] If $G$ is a finite {\'e}tale $S$-group scheme, then $G$ is the N{\'e}ron model of its generic fiber.
\end{enumerate}
\end{lem}

\begin{proof}
1) Let us note that, $p$ being invertible on $S$, the $S$-group scheme $\E[p]$ is {\'e}tale, and is a closed subgroup scheme of $\E$. Therefore, $\E[p]$ is the N{\'e}ron model of $E[p]$, according to \cite[\S{}7.1, Corollary~6]{NeronModels}.

2) By hypothesis, $G$ is a smooth and separated $S$-group scheme. Moreover, the map $G\to S$ is finite, hence proper. It follows from \cite[\S{}7.1, Theorem~1]{NeronModels} that $G$ is the N{\'e}ron model of its generic fiber.
\end{proof}

Let $\Res_{k(C)/k(S)}\mu_p$ denote the Weil restriction of $\mu_p$, which, by definition, satisfies
$$
(\Res_{k(C)/k(S)}\mu_p)(L)=\mu_p(L\otimes k(C))
$$
for any $k(S)$-algebra $L$. Let $T$ denote the generic point of $C$, which is also the generic non-zero $p$-torsion point of $E$. Then the Weil pairing $e_p$ induces a morphism of finite {\'e}tale $k(S)$-group schemes (equivalently of Galois modules over $k(S)$)
\begin{align*}
w:E[p] &\longrightarrow \Res_{k(C)/k(S)}\mu_p \\
P &\longmapsto e_p(P,T).
\end{align*}

According to Lemma~\ref{unramified}, $\E[p]$ and $\mu_p$ are N{\'e}ron models of their generic fibers. It follows from \cite[\S{}7.6, Proposition~6]{NeronModels} that the same holds for $\Res_{C/S}\mu_p$.
We denote by
$$
\underline{w}:\E[p] \longrightarrow \Res_{C/S}\mu_p
$$
the map induced by $w$ on these N{\'e}ron models.

\begin{lem}
\label{w_injective}
The map $\underline{w}$ induces on {\'e}tale cohomology an injective map
$$
h^1\underline{w}:H^1(S,\E[p]) \longrightarrow H^1(S,\Res_{C/S}\mu_p)=H^1(C,\mu_p),
$$
whose image is contained in the kernel of the norm map
$$
N_{C/S}:H^1(C,\mu_p)\longrightarrow H^1(S,\mu_p).
$$
When $p\geq 3$, the image of $h^1\underline{w}$ is contained in the kernel of the norm map
$$
N_{C/C^+}:H^1(C,\mu_p)\longrightarrow H^1(C^+,\mu_p).
$$
\end{lem}

\begin{proof}
It has been proved by Djabri, Schaefer and Smart \cite[Prop.~7 and Prop.~8]{dss} that the map $w$ induces on Galois cohomology an injective map
$$
h^1w:H^1(k(S),E[p]) \longrightarrow H^1(k(S),\Res_{k(C)/k(S)}\mu_p)=H^1(k(C),\mu_p).
$$

Let us stress here the fact that the authors cited above work over a number field or the completion of a number field. Nevertheless, their arguments are purely Galois-theoretical and extend without any change to our function field setting.

Now, consider the following commutative diagram
$$
\begin{CD}
H^1(S,\E[p]) @>h^1\underline{w}>> H^1(C,\mu_p) \\
@VVV @VVV \\
H^1(k(S),E[p]) @>h^1w>> H^1(k(C),\mu_p)\\
\end{CD}
$$
in which the vertical maps are injective, according to \eqref{injNeron}. Thus, one deduces the injectivity of $h^1\underline{w}$ from that of $h^1w$. Then, we note that the composition of the maps
$$
\begin{CD}
E[p] @>w>> \Res_{k(C)/k(S)}\mu_p @>N_{k(C)/k(S)}>> \mu_p \\
\end{CD}
$$
is zero, because the sum of all non-zero elements of $E[p]$ is zero. It follows that the composition of the induced maps on the N{\'e}ron models is also zero, hence the same holds for the maps induced on cohomology. Finally, when $p\geq 3$, the map $P\mapsto -P$ is a nontrivial involution of $C$, with quotient $C^+$, and the composition of the maps
$$
\begin{CD}
E[p] @>w>> \Res_{k(C)/k(S)}\mu_p @>N_{k(C)/k(C^+)}>> \Res_{k(C^+)/k(S)}\mu_p \\
\end{CD}
$$
is zero, because $e_p(P,T)\cdot e_p(-P,T)=1$. This proves the last part of the statement.
\end{proof}

\begin{proof}[Proof of Theorem~\ref{thm_kummer}]
Let us note that the map $[p]:\E^0_v\to \E^0_v$ is surjective for the {\'e}tale topology on $k_v$ (multiplication by $p$ on a smooth connected group scheme in characteristic $\neq p$ is {\'e}tale surjective). It follows that multiplication by $p$ on $\E$ induces an exact sequence for the {\'e}tale topology on $S$
$$
\begin{CD}
0 @>>> \E[p] @>>> \E @>[p]>> \E^{p\Phi} @>>> 0. \\
\end{CD}
$$

Applying cohomology to this sequence, we obtain an injective map
$$
\delta:\E^{p\Phi}(S)/p\E(S) \longrightarrow H^1(S,\E[p]).
$$

According to Lemma~\ref{w_injective}, the target group can be embedded in the kernel of the norm map
$$
N_{C/S}:H^1(C,\mu_p)\longrightarrow H^1(S,\mu_p).
$$

Let us determine the size of this kernel.
The prime $p$ being invertible in $k$, multiplication by $p$ on $\Gm$ is surjective for the {\'e}tale topology on $k$-schemes. Moreover, $S$ and $C$ being geometrically integral projective curves over $k$, we have $\Gm(S)=\Gm(C)=k^{\times}$. Therefore, Kummer theory over $S$ and $C$ yields a commutative diagram with exact rows
$$
\begin{CD}
0 @>>> k^{\times}/(k^{\times})^p @>>> H^1(C,\mu_p) @>\kappa>> \Pic(C)[p] @>>> 0 \\
@. @| @VVN_{C/S}V @VVN_{C/S}V \\
0 @>>> k^{\times}/(k^{\times})^p @>>> H^1(S,\mu_p) @>>> \Pic(S)[p] @>>> 0 \\
\end{CD}
$$
where the vertical maps are the norm maps.  We deduce, by the snake lemma, that the two norm maps $N_{C/S}$ above have the same kernel. In other words, if we compose the natural map $\kappa:H^1(C,\mu_p)\to\Pic(C)[p]$ with $h^1\underline{w}\circ\delta$, we obtain an injective group morphism
$$
\E^{p\Phi}(S)/p\E(S) \longrightarrow \ker\left(N_{C/S}:\Pic(C)[p]\to\Pic(S)[p]\right)
$$
which proves \eqref{eq:2descent_map} when $p=2$. When $p\geq 3$, the last statement of Lemma~\ref{w_injective} allows us to prove \eqref{eq:pdescent_map} with a similar argument. When $p=2$, $C\to S$ has degree $3$, hence $N_{C/S}$ is surjective on $2$-torsion, from which we deduce that
$$
\dim_{\F_2} \ker\left(N_{C/S}:\Pic(C)[2]\to\Pic(S)[2]\right) = \dim_{\F_2} \Pic(C)[2] - \dim_{\F_2} \Pic(S)[2].
$$
Similarly, when $p\geq 3$ the map $C\to C^+$ has degree $2$ and we have
$$
\dim_{\F_p} \ker\left(N_{C/C^+}:\Pic(C)[p]\to\Pic(C^+)[p]\right) = \dim_{\F_p} \Pic(C)[p] - \dim_{\F_p} \Pic(C^+)[p]
$$
The statements \eqref{inequality2} and \eqref{inequality1} follow by combining this with Lemma~\ref{rank_size}.
\end{proof}

\begin{rmq}
\label{rmq:geometricSelmer}
In \cite[\S{}4.2]{CSSD98}, the authors define a geometric Selmer group, denoted $\mathfrak{S}(\E,p)$, which, according to Prop.~4.2.2 of \emph{loc.~cit.}, fits into an exact sequence
$$
\begin{CD}
0 @>>> H^1(S,\E[p]) @>>> \mathfrak{S}(\E,p) @>>> H^0(S,\Phi/p\Phi). \\
\end{CD}
$$
Our proof of Theorem~\ref{thm_kummer} relies on a) bounding the size of $H^1(S,\E[p])$ in terms of the curves $S$ and $C$, and b) computing the exact size of $H^0(S,\Phi/p\Phi)$. This yields an upper bound on the size of the geometric Selmer group.
\end{rmq}

\begin{rmq}
It is important to require that $C$ and $S$ are both geometrically integral and projective.  The fact that $\Gm(S)=\Gm(C)$ is a crucial argument in the proof, which implies that the units do not contribute to the kernel of the norm.   Recall from Remark~\ref{rmq:nongeneric} that, when the full $p$-torsion is defined over $k(S)$, it may happen that the geometric analogue of the Selmer group is infinite.
\end{rmq}

\begin{rmq}
\label{rmq:Dokchitser}
In the framework of $p$-descent over a number field, it was proved by Dokchitser \cite[Cor.~6.5.2]{Dok2000}, under the same assumption (the Galois action on $E[p]$ is transitive), that the image of $h^1\underline{w}$ is contained in the kernel of the norm $N_{k(C)/K}$ for any proper subfield $K\subset k(C)$. So one can improve the statement of Theorem~\ref{thm_kummer}, 1) as follows: if $p\geq 3$, we have an injective morphism
$$
\E^{p\Phi}(S)/p\E(S) \longrightarrow \bigcap_{C\to C'} \ker\left(N_{C/C'}:\Pic(C)[p]\to\Pic(C')[p]\right)
$$
where $C\to C'$ runs through all proper subcovers of $C\to S$. In fact, for $p=3$ this is not an improvement since every proper subcover of $C\to S$ factors through $C\to C^+$ (see the proof of Theorem~\ref{thm:3arithmeticIgusa}).
\end{rmq}


\subsection{Proof of Theorem~\ref{thm:arithmeticIgusa}}
\label{sub:Igusa}


Let us recall that, if $f:X\to Y$ is a finite flat, tamely ramified map of smooth $k$-curves, then the ramification divisor of $f$ is given by
$\sum_x (e_x-1)\cdot x$,
where $x$ runs through closed points of $X$, and $e_x$ denotes the ramification index of $f$ at $x$. All finite maps we shall consider here are tamely ramified.

By assumption, $\car(k)\neq 2, 3$ and $\E$ does not have a nontrivial $2$-torsion section over $S$.
In order to prove \eqref{eq:2descentkbar}, we may assume that $k$ is algebraically closed, which we do until further notice.

Let $C$ be the smooth compactification of $\E[2]\setminus \{0\}$. We note that $C\to S$ is tamely ramified, because the Galois closure of its generic fiber has Galois group $\Z/3\Z$ or $\mathfrak{S}_3$, whose order is coprime to the characteristic of $k$.

We shall first compute the degree of the ramification divisor of $C\to S$. 
 
\begin{lem}
\label{lem:C2ramif}
Let $R\subset C$ be the ramification divisor of the natural cubic map $C\to S$. Then the degree of $R$ is given by the formula
$$
\deg(R)=\deg(\mathfrak{f}_E)-\sum_{v\in\Sigma} \varepsilon_v,
$$
where
$$
\varepsilon_v = \left\{
\begin{array}{ll}
2 & \text{if reduction type $\mathrm{I}_{2n}^*$ for some $n\geq 0$}\\
1 & \text{if $2\mid c_v$, other reduction type}\\
0 & \text{otherwise.}
\end{array}\right.
$$
\end{lem}

\begin{rmq}
Note that $\sum_{v\in\Sigma} \varepsilon_v = \dim_{\F_2} H^0(S,\Phi[2])$.
\end{rmq}

\begin{proof}
By construction, $\E[2]\setminus \{0\}$ is an open subscheme of $C$. In fact, we claim that
\begin{equation}
\label{eq:ramif}
C\setminus R=\E[2]\setminus \{0\}.
\end{equation}

Let us prove the two inclusions: the right hand side is a subscheme of the left hand side, because $\E[2]\setminus \{0\}\to S$ is {\'e}tale (the assumption that $\car(k)\neq 2$ is important here). The other inclusion follows from the universal property of the N{\'e}ron model: $C\setminus R\to S$ is a smooth separated scheme over $S$, hence the inclusion of the generic fiber $E[2]\setminus \{0\}\hookrightarrow E$ can be extended to a map $C\setminus R\to \E$, which takes values in $\E[2]\setminus \{0\}$.

If $v\in\Sigma$ is a place of bad reduction, we denote by $R_v$ the fiber of $R$ above $v$. The map $C\to S$ being finite flat of degree $3$, it follows from \eqref{eq:ramif} that
\begin{equation}
\label{eq:degRv}
\deg(R_v)=\left\{
\begin{array}{ll}
0 & \text{if $\# \E_v[2]=4$} \\
1 & \text{if $\# \E_v[2]=2$} \\
2 & \text{if $\# \E_v[2]=1$},
\end{array}\right.
\end{equation}
where $\# \E_v[2]$ denotes the number of $k$-points of $\E_v[2]$, or equivalently, the rank of $\E_v[2]$ as a finite $k$-group scheme (remember that $k$ is algebraically closed here).

On the other hand, the map $[2]:\E_v^0 \to\E_v^0$ being surjective for the {\'e}tale topology, we deduce from \eqref{eq:compgroup} a short exact sequence of {\'e}tale sheaves
$$
\begin{CD}
0 @>>> \E_v^0[2] @>>> \E_v[2] @>>> \Phi_v[2] @>>> 0. \\
\end{CD}
$$

It follows from the explicit description of the bad fibers of the N{\'e}ron model of an elliptic curve \cite[Chap.~IV, Table~4.1]{silvermanII} that
$$
\# \E_v[2]=\left\{
\begin{array}{ll}
4 & \text{if reduction type $\mathrm{I}_{2n}^*$ for some $n\geq 0$}\\
4 & \text{if reduction type $\mathrm{I}_{2n}$ for some $n\geq 0$}\\
2 & \text{if reduction type $\mathrm{I}_{2n+1}$ for some $n\geq 0$}\\
2 & \text{if reduction type $\mathrm{III}$, $\mathrm{III}^*$ or $\mathrm{I}_{2n+1}^*$ for some $n\geq 0$} \\
1 & \text{if reduction type $\mathrm{II}$, $\mathrm{II}^*$, $\mathrm{IV}$ or $\mathrm{IV}^*$.}
\end{array}\right.
$$

For example, in the case of reduction $\mathrm{I}_{2n}^*$, the group $\Phi_v$ is isomorphic to $(\Z/2\Z)^2$, hence $\E_v[2]$ has four points. In the case of multiplicative reduction $\mathrm{I}_{2n+1}$, we have $\Phi_v\simeq \Z/(2n+1)$ and $\E_v^0\simeq \Gm$, hence $\E_v[2]=\E_v^0[2]=\mu_2$, which has two points.

Using \eqref{eq:degRv} one checks that, in each case listed above,
$$
\deg(R_v) = f_v - \varepsilon_v,
$$
where $f_v$ is the exponent of the conductor at $v$. The result follows immediately by summing up over all $v\in\Sigma$.
\end{proof}

\begin{proof}[Proof of Theorem~\ref{thm:arithmeticIgusa}]
Applying the Riemann-Hurwitz formula to the cubic map $C\to S$, we find that
$$
2g(C)-2=3(2g(S)-2)+\deg(R),
$$
where $R\subset C$ is the ramification divisor of $C\to S$. Equivalently,
$$
2g(C)-2g(S)=4g(S)-4+\deg(R).
$$
Replacing $\deg(R)$ by its value computed in Lemma~\ref{lem:C2ramif} yields
$$
2g(C)-2g(S) + \sum_{v\in\Sigma} \varepsilon_v = 4g(S)-4 +\deg(\mathfrak{f}_E).
$$
But, by definition of $\varepsilon_v$, we have
$$
\sum_{v\in\Sigma} \varepsilon_v = \#\{v\in \Sigma, 2\mid c_v\} +\#\{v\in \Sigma, \text{the red. type of $\E$ at $v$ is $\mathrm{I}_{2n}^*$ for some $n\geq 0$} \}
$$
which proves \eqref{eq:2descentkbar}.

Let us now prove the last statement of Theorem~\ref{thm:arithmeticIgusa}. Keeping for the moment the assumption that $k$ is algebraically closed, we have
$$
\dim_{\F_2}\Pic(C)[2]=2g(C)
$$
and similarly for $S$. Therefore, the relation \eqref{eq:2descentkbar} means that \eqref{inequality2} is equivalent to the geometric rank bound. Dropping now the assumption that $k$ is algebraically closed, it follows that the bound \eqref{inequality2} is a refinement of the geometric rank bound.
\end{proof}


\subsection{Proof of Theorem~\ref{thm:0bound3}}
\label{sub:0bound}


We keep the notation of the previous section.

\begin{proof}[Proof of Theorem~\ref{thm:0bound3}]

We first prove the bound \eqref{imp1}.  Let $\overline{\Sigma}$ be the set of places of bad reduction of $\E_{\Qbar}$ over $S_{\Qbar}$.  From the proof of Theorem \ref{thm:arithmeticIgusa}, we have 
\begin{align*}
g(C)=3g(S)-2+\frac{1}{2}\left(\deg(\mathfrak{f}_E)-\sum_{v\in\overline{\Sigma}} \varepsilon_v\right).
\end{align*}
Since $C$ has good reduction at $3$, the reduction map modulo $3$ is injective on $2$-torsion, and Theorem~\ref{thm:Bhargava} implies that
\begin{align*}
\dim_{\F_2} \Pic(C)[2] &\leq \log_2\left(\frac{3^{g(C)+1}-1}{2}\right)< (g(C)+1)\log_23-1\\
&<(\log_23)\left(3g(S)-1+\frac{1}{2}\left(\deg(\mathfrak{f}_E)-\sum_{v\in\overline{\Sigma}}\varepsilon_v\right)\right)-1.
\end{align*}
For $v\in \overline{\Sigma}$, let $m_v$ denote the number of irreducible components in the fiber above $v$.  Then it follows easily from the definition of $\varepsilon_v$ and a known formula for $12\chi$ \cite{shioda92} that
\begin{align*}
\deg(\mathfrak{f}_E)+\sum_{v\in\overline{\Sigma}}\varepsilon_v\leq \deg(\mathfrak{f}_E)+\sum_{v\in\overline{\Sigma}}(m_v-1)=12\chi,
\end{align*}
and this inequality is sharp only if $\varepsilon_v=m_v-1$ for each bad $v$.

Let us now consider the inequality \eqref{inequality2}, in which we neglect the negative term $-\dim_{\F_2} \Pic(S)[2]$, and observe that $\sum_{v\in\Sigma}\varepsilon_v\leq \sum_{v\in\overline{\Sigma}}\varepsilon_v$, which yields:
\begin{align*}
\rk_{\Z} E(\Q(S))&\leq \dim_{\F_2} \Pic(C)[2]+\sum_{v\in\overline{\Sigma}}\varepsilon_v\\
&\leq(\log_23)\left(3g(S)-1+\frac{1}{2}\left(\deg(\mathfrak{f}_E)-\sum_{v\in\overline{\Sigma}}\varepsilon_v\right)\right)-1+\sum_{v\in\overline{\Sigma}}\varepsilon_v\\
&= 3(\log_23)g(S)+\frac{\log_23}{2}\deg(\mathfrak{f}_E)+\left(1-\frac{\log_23}{2}\right)\sum_{v\in\overline{\Sigma}}\varepsilon_v -\log_23-1\\
&= 3(\log_23)g(S)+\frac{\log_23}{2}\left(\deg(\mathfrak{f}_E)+\sum_{v\in\overline{\Sigma}}\varepsilon_v\right)-(\log_23-1)\sum_{v\in\overline{\Sigma}}\varepsilon_v-\log_23-1\\
&\leq 3(\log_23)g(S)+6 (\log_23)\chi -(\log_23-1)\sum_{v\in\overline{\Sigma}}\varepsilon_v -\log_23-1.
\end{align*}

For the bound \eqref{imp2}, suppose that $C$ has good reduction at $p\in \{3,5\}$.  Since $S=\mathbb{P}^1$, the curve $C$ is trigonal.  It is well-known that the gonality of $C$ can only decrease after reduction modulo a good prime, and so applying Theorem~\ref{thm:Bhargava} with $n=3$ yields
\begin{align*}
\dim_{\F_2} \Pic(C)[2] &\leq \frac{2}{3}g(C)\log_2p+O(1).
\end{align*}
Now nearly the exact same calculation as above gives
\begin{align*}
\rk_{\Z} E(\Q(t))\leq 4(\log_2 p)\chi+O(1).
\end{align*}
\end{proof}


\subsection{Proof of Theorem~\ref{thm:3arithmeticIgusa}}
\label{sub:3Igusa}


By assumption, $\car(k)\neq 2, 3$ and the prime $3$ satisfies assumptions (i)--(iii).
In order to prove \eqref{eq:3descentkbar}, we may assume that $k$ is algebraically closed, which we do until further notice.

Let $C$ be the smooth compactification of $\E[3]\setminus \{0\}$, and let $C^+$ be the quotient of $C$ by the involution $P\mapsto -P$.

The composite morphism $\Gal(k(S)(E[3])/k(S)) \to \GL_2(\F_3) \stackrel{\mathrm{det}}{\longrightarrow} \F_3^*$ is the cyclotomic character, which is trivial since $k$ is algebraically closed. Therefore, $\Gal(k(S)(E[3])/k(S))$ is a subgroup of $\SL_2(\F_3)$ which, by assumption (i)--(iii), acts transitively on $\F_3^2\setminus \{0\}$. It follows that $\Gal(k(S)(E[3])/k(S))$ is equal to $\SL_2(\F_3)$ or to its $2$-Sylow subgroup the quaternionic group $Q_8$, which is normal in $\SL_2(\F_3)$, with quotient group cyclic of order three.

In fact, it is easy to check \cite[Prop.~5.4.3]{Adelmann} that the subfield of $k(S)(E[3])$ fixed by $Q_8$ is $k(S)(\sqrt[3]{\Delta})$, where $\Delta \in k(S)$ is the discriminant of $E$. Therefore, $\Gal(k(S)(E[3])/k(S))=Q_8$ if and only if $\Delta$ is a cube in $k(S)$. If not, then $k(S)(E[3])$ is the compositum of $k(C)$ and $k(S)(\sqrt[3]{\Delta})$ over $k(S)$. Switching to geometry, if we denote by $S[\sqrt[3]{\Delta}]$ (resp. $C[\sqrt[3]{\Delta}]$) the curve with function field $k(S)(\sqrt[3]{\Delta})$ (resp. $k(S)(E[3])$), we have the following cartesian square of connected covers of curves, in which Galois covers are labelled with their corresponding Galois group.


\newcommand{\mypicture}{
\begin{tikzpicture}
  \node (Q1) at (0,0) {$S$};
  \node (Q2) at (2,2) {$C$};
  \node (Q3) at (0,4) {$C[\sqrt[3]{\Delta}]$};
  \node (Q4) at (-2,2) {$S[\sqrt[3]{\Delta}]$};

  \draw (Q1)--(Q2) node [pos=0.7, below,inner sep=0.3cm] {8};
  \draw (Q1)--(Q4) node [pos=0.9, below,inner sep=0.5cm] {$\Z/3\Z$};
  \draw (Q3)--(Q4) node [pos=0.8, above,inner sep=0.5cm] {$Q_8$};
  \draw (Q2)--(Q3) node [pos=0.2, above,inner sep=0.5cm] {$\Z/3\Z$};
  \draw (Q1)--(Q3) node [pos=0.5, right,inner sep=0.1cm] {$\SL_2(\F_3)$};
\end{tikzpicture}
} 

\begin{wrapfigure}{r}{0.3\textwidth}  \mypicture{}  \end{wrapfigure}

Let us underline the fact that, in any case, $C[\sqrt[3]{\Delta}]\to S$ is tamely ramified, since it is Galois of degree $24$ (resp. $8$ in the quaternionic case), which is coprime to the characteristic of $k$. Therefore, the inertia groups of its ramified points are cyclic subgroups of $\SL_2(\F_3)$, hence have order $2$, $3$, $4$ or $6$ (resp. $2$ or $4$ in the quaternionic case).

Let $R\subset C$ be the ramification divisor of $C\to S$. If $v\in S$ is a closed point, we denote by $C_v$ (resp. by $R_v$) the fiber of $C$ (resp. $R$) at $v$. As previously \eqref{eq:ramif}, we note that
$$
C\setminus R=\E[3]\setminus \{0\}.
$$

Similarly, we denote by $R^+\subset C^+$ the ramification divisor of $C^+\to S$. We note that $Q_8$ has a unique subgroup of order two, which means that $P\mapsto -P$ is the unique automorphism of order $2$ of $C[\sqrt[3]{\Delta}]\to S$. The same holds for $C\to S$ (which is not Galois in general, but over which the automorphism $P\mapsto -P$ is still defined). It then follows from the diagram above that, given a point $P\in C$ with ramification index $e_P$ (relative to $S$), its image by the map $C\to C^+$ has ramification index $e_P/2$ if $e_P$ is even, and $e_P$ if $e_P$ is odd.

We shall now proceed to a case-by-case description of $R_v$ and $R_v^+$ depending on the reduction type of $\E$ at $v$.
\begin{enumerate}
\item Reduction type $\mathrm{I}_n$, $3\mid n$. In this case, $\#\E_v[3]=9$ hence $R_v$ and $R_v^+$ are both zero.
\item Reduction type $\mathrm{I}_n$, $3\nmid n$. In this case, $\#\E_v[3]=3$ hence $C_v$ has exactly two unramified points $P_1$ and $P_2$. Moreover, $v(\Delta)=n$, hence $\Delta$ is not a cube in $k(S)$ and, after performing the totally ramified base change $S[\sqrt[3]{\Delta}] \to S$, $E$ has reduction type $I_{3n}$, in particular $C[\sqrt[3]{\Delta}]\to S[\sqrt[3]{\Delta}]$ is unramified above $v$ (previous case). It follows that the ramified points of $C\to S$ have ramification index $3$, hence $C_v=P_1+P_2+3Q_1+3Q_2$ and $C^+_v=P^++3Q^+$. Therefore, $\deg(R_v)=4$ and $\deg(R_v^+)=2$.
\item Reduction type $\mathrm{I}_n^*$, $3\mid n$. In this case, $\#\E_v[3]=1$ hence $C_v$ has all its points ramified. Moreover, $v(\Delta)=6+n$, hence $S[\sqrt[3]{\Delta}] \to S$ is unramified above $v$. So without loss of generality we may assume that $\Delta$ is a cube in $k(S)$, i.e. $C\to S$ is Galois with group $Q_8$. On the other hand it is known that, after a quadratic ramified base change, $E$ has reduction type $I_{2n}$, in which case $C\to S$ is unramified (first case). We deduce that all ramification indexes of $C\to S$ are equal to $2$, i.e. $C_v= 2P_1+2P_2+2P_3+2P_4$, and that $C^+_v=P_1^++P_2^++P_3^++P_4^+$. Therefore, $\deg(R_v)=4$ and $\deg(R_v^+)=0$.
\item Reduction type $\mathrm{I}_n^*$, $3\nmid n$. As in the previous case, $C_v$ has all its points ramified. After performing a quadratic ramified base change, $E$ has reduction type $I_{2n}$, and we recover the case 2. We conclude that $C_v=2P+6Q$ and $C^+_v=P^++3Q^+$. Therefore, $\deg(R_v)=6$ and $\deg(R_v^+)=2$.
\item Reduction type $\mathrm{IV}$ or $\mathrm{IV}^*$. In this case, $\#\E_v[3]=3$ hence $C_v$ has exactly two unramified points $P_1$ and $P_2$. Moreover, $v(\Delta)$ is $4$ or $8$, hence $\Delta$ is not a cube in $k(S)$ and, after performing the totally ramified base change $S[\sqrt[3]{\Delta}] \to S$, $E$ acquires good reduction, because the valuation of $\Delta$ is multiplied by $3$, hence is zero modulo $12$. Therefore, $C[\sqrt[3]{\Delta}]\to S[\sqrt[3]{\Delta}]$ is unramified above $v$. We conclude that the ramification is the same as in 2.
\item[] In the remaining cases, the reduction is additive, potentially good, and $C_v$ has all its points ramified. According to Serre-Tate \cite[Cor.~2]{SerreTate}, the curve $E$ acquires good reduction exactly when the field of definition of the $3$-torsion points becomes unramified. Therefore, the ramification indexes of  $C[\sqrt[3]{\Delta}]\to S$ are all equal to the \textit{semistability defect} of $E$, which is the denominator of $v(\Delta)/12$ since $\car(k)\neq 2, 3$ \cite[Prop.~1]{Kraus1990}.
\item Reduction type $\mathrm{II}$ or $\mathrm{II}^*$. In this case, $C_v = 2P+6Q$, $\deg(R_v)=6$ and $\deg(R_v^+)=2$.
\item Reduction type $\mathrm{III}$ or $\mathrm{III}^*$. In this case, $C_v = 4P+4Q$, $\deg(R_v)=6$ and $\deg(R_v^+)=2$.
\end{enumerate}

One checks that, in each case listed above,
$$
\frac{1}{2}\left(\deg(R_v) - \deg(R_v^+)\right) = f_v -
\left\{
\begin{array}{ll}
1 & \text{if $3\mid c_v$}\\
0 & \text{otherwise.}
\end{array}\right.
$$
Summing up over all $v\in\Sigma$, it follows that
\begin{equation}
\label{eq:degR-degR+}
\frac{1}{2}\left(\deg(R)-\deg(R^+)\right) = \deg(\mathfrak{f}_E) - \#\{v\in \Sigma, 3\mid c_v\}.
\end{equation}

The proof ends by applying the Riemann-Hurwitz formula to the covers $C\to S$ of degree $8$, and $C^+\to S$ of degree $4$. Subtracting the two identities and dividing by $2$, one obtains
$$
g(C)-g(C^+) = 4g(S) - 4 +\frac{1}{2}\left(\deg(R)-\deg(R^+)\right),
$$
and \eqref{eq:3descentkbar} follows by combining this with \eqref{eq:degR-degR+}.

Let us now prove the last statement of Theorem~\ref{thm:3arithmeticIgusa}. Keeping for the moment the assumption that $k$ is algebraically closed, we have $\dim_{\F_3}\Pic(C)[3]=2g(C)$, and similarly for $C^+$. By self-duality of the Jacobian of $C$, the Weil pairing induces a perfect, symplectic pairing $\Pic(C)[3]\times \Pic(C)[3]\to \mu_3$, which is Galois equivariant. Dropping now the assumption that $k$ is algebraically closed, it follows that in particular, if $k$ does not contain cube roots of unity, then
$$
\dim_{\F_3} \Pic(C)[3] \leq g(C)
$$
and similarly for $C^+$. Therefore, the identity \eqref{eq:3descentkbar} implies that, for $p=3$, the bound \eqref{inequality1} is a refinement of the geometric rank bound provided $k$ does not contain cube roots of unity.



\bibliographystyle{amsalpha}
\bibliography{biblio}



\bigskip

\textsc{Jean Gillibert}, Institut de Math{\'e}matiques de Toulouse, CNRS UMR 5219, 118 route de Narbonne, 31062 Toulouse Cedex 9, France.

\emph{E-mail address:} \texttt{jean.gillibert@math.univ-toulouse.fr}
\medskip

\textsc{Aaron Levin}, Department of Mathematics, Michigan State University, 619 Red Cedar Road, East Lansing, MI 48824, USA.

\emph{E-mail address:} \texttt{adlevin@math.msu.edu}


\end{document}